\documentclass[a4paper,american,reqno]{amsart}

\newif\ifPreprint \Preprinttrue
\newif\ifSubmission \Submissionfalse

\usepackage{babel}
\usepackage[utf8]{inputenc}
\usepackage[T1]{fontenc}
\usepackage{lmodern}
\usepackage{natbib}
\usepackage{xspace}
\usepackage{mathtools}
\usepackage{amsmath}
\usepackage{enumitem}
\usepackage{accents}
\usepackage{todonotes}
\usepackage{csquotes}
\usepackage{microtype}
\usepackage[colorlinks,
            citecolor=blue,
            urlcolor=black,
            linkcolor=blue]{hyperref}
\usepackage{cleveref}

\makeatletter
\patchcmd{\@settitle}{\uppercasenonmath\@title}{\scshape\large}{}{}
\patchcmd{\@setauthors}{\MakeUppercase}{\scshape\normalsize}{}{}
\makeatother

\vfuzz \hfuzz
\theoremstyle{plain}
\newtheorem{theorem}{Theorem}[section]
\newtheorem{lemma}[theorem]{Lemma}
\newtheorem{corollary}[theorem]{Corollary}

\theoremstyle{definition}

\newtheorem{example}[theorem]{Example}

\newtheorem{remark}[theorem]{Remark}

\theoremstyle{remark}

\newcommand{\defsep}{\colon}
\newcommand{\Unc}{\mathcal{U}}

\newcommand{\st}{\text{s.t.}}

\newcommand{\C}{\mathcal{C}}
\newcommand{\X}{\mathcal{X}}
\newcommand{\R}{\mathbb{R}}

\newcommand{\Nc}{\mathcal{N}}
\newcommand{\T}{\mathcal{T}}

\newcommand{\tu}{\tilde{u}}

\newcommand{\tauU}{\tau({\Unc})}
\newcommand{\fa}{\text{ for all }}
\newcommand{\varc}{c^{\text{var}}}
\newcommand{\invc}{c^{\text{inv}}}
\newcommand{\boxv}{\text{val}_{\mathcal{B}}}
\newcommand{\robv}{\text{val}_{\text{R}}}
\newcommand{\ARv}{\text{val}_{\text{AR}}}
\newcommand{\crobv}{\text{cval}_{\text{R}}}
\newcommand{\cARv}{\text{cval}_{\text{AR}}}
\newcommand{\boxtv}{\text{val}_{\tilde{\mathcal{B}}}}
\newcommand{\bdt}{\boldsymbol{\cdot}}

\newcommand{\RobWel}{C_R}
\newcommand{\ARobWel}{C_{AR}}

\DeclareMathOperator*{\argmin}{arg\,min}

\DeclareMathOperator{\diag}{diag}
\DeclareMathOperator{\mvar}{MVaR}
\DeclareMathOperator{\var}{VaR}

\newcommand{\rev}[1]{#1} 
\newcommand{\revv}[1]{#1} 

\begin{document}

\title[Robust Market Equilibria under Uncertain Cost]%
{Robust Market Equilibria under Uncertain Cost}


\author[C. Biefel, F. Liers, J. Rolfes, L. Schewe, G. Zöttl]%
{Christian Biefel, Frauke Liers, Jan Rolfes, \\Lars Schewe, Gregor Zöttl}

\address[C. Biefel, F. Liers, J. Rolfes]{%
  Friedrich-Alexander-Universität Erlangen-Nürnberg,
  Germany, Department of Data Science, Cauerstr. 11, 91058 Erlangen;
  Energie Campus Nürnberg,
  Fürther Str. 250,
  90429 Nürnberg,
  Germany}
\email{\{christian.biefel,frauke.liers,jan.rolfes\}@fau.de}

\address[L. Schewe]{%
University of Edinburgh,
  School of Mathematics and Maxwell Institute for Mathematical Sciences,
  James Clerk Maxwell Building,
  Peter Guthrie Tait Road,
  Edinburgh, EH9 3FD, UK;
  Energie Campus Nürnberg,
  Fürther Str. 250,
  90429 Nürnberg,
  Germany}
\email{lars.schewe@ed.ac.uk}

\address[G. Zöttl]{%
  Friedrich-Alexander-Universität Erlangen-Nürnberg,
  Germany, Institute of Economic Research, Lange Gasse 20, 90403 Nürnberg;
  Energie Campus Nürnberg,
  Fürther Str. 250,
  90429 Nürnberg,
  Germany}
\email{gregor.zoettl@fau.de}

\date{\today}

\begin{abstract}
\revv{This work studies} equilibrium problems under uncertainty where firms
maximize their profits in a robust way when selling their output. Robust optimization plays an increasingly important role when best guaranteed objective values are to be determined, independently of the specific distributional assumptions regarding uncertainty. In particular, solutions are to be determined that are feasible regardless of how the uncertainty manifests itself within some predefined uncertainty set. 
Our \rev{mathematical} analysis adopts the robust optimization perspective in the context of equilibrium problems. First, we \rev{present structural insights for} a single-stage,  nonadjustable robust setting. We then go one step further and study the more complex two-stage or adjustable case where a part of the variables can adjust to the realization of the uncertainty. We compare equilibrium outcomes with the corresponding centralized
robust optimization problem where the sum of all profits are maximized. As we find, the market equilibrium for the perfectly competitive firms differs from the solution of the robust central planner, which is in stark contrast to classical results regarding the efficiency of market equilibria with perfectly competitive firms. For the different scenarios considered, we furthermore are able to determine the resulting price of anarchy. In the case of non-adjustable robustness, for fixed demand in every time step the price of anarchy is bounded whereas it is unbounded if the buyers are modeled by elastic demand functions. For the two-stage adjustable setting, we show how to compute subsidies for the firms that lead to robust welfare optimal equilibria. \\


\noindent \textsc{Corresponding author.} Christian Biefel - christian.biefel@fau.de

\end{abstract}

\maketitle

\section{Introduction}
\label{sec:introduction}

Equilibrium problems arise whenever several agents which interact in a common context maximize their own objectives. An equilibrium then corresponds to a 'stable' situation in which no agent has an incentive to deviate from his or her optimum strategy, given the strategies of all other agents. Such tasks appear in wide contexts, for example in market interactions, in transportation, and the like. Many insights have already been obtained for the case of deterministic and risk neutral settings. In many situations however, uncertainty and risk play an essential role for the optimization problems of the different agents and the resulting equilibria. Indeed, when neglecting the influence of uncertainty, solutions may differ considerably, and so it is advisable to hedge against them. 

Recently, the explicit consideration of uncertainties in equilibrium problems has obtained increased attention. Most of those contributions model the uncertainties from a stochastic optimization perspective. Stochastic optimization~(\cite{Birge_Louveaux:2011}) protects against uncertainties with a certain probability and in expectation. A prominent approach, for example, consists in the introduction of risk functions which allows to take into account risk aversion (for the seminal work compare \cite{Artzneretal:1999}).  

A stochastic optimization approach naturally assumes that information on the underlying probability distributions is available. In many settings, however, such knowledge may not be available, or, in safety-critical situations, protection in a probabilistic sense may not be enough. Furthermore, the resulting problems are not necessarily algorithmically tractable. In such situations, a different modelling approach is appropriate that has its foundations in robust optimization.

In robust optimization~(\cite{Ben-Tal_et_al:2009}), the task is to ensure feasibility of a solution regardless of how uncertainties manifest themselves within predefined uncertainty sets. Among all such robust feasible solutions, a robust optimum is determined that yields best guaranteed objective values. A main challenge in this very active research area consists in modelling or reformulating the robust counterpart problem in a way that leads to algorithmical tractability. Another topic of research studies how conservative the robust solutions are, i.e., the cost of robust protection when compared to the unprotected case in which uncertainties are ignored. 

In the present article, we thus \rev{present structural insights into} equilibrium problems where firms in a market context maximize their profits in a robust way when selling their output. Firms optimize their investment and production decisions for several periods of time facing uncertain production costs. We first \rev{analyze} the single-stage or non-adjustable robust setting, where the firms' production decisions are determined by the original output decision for all periods of time and for all realizations of cost uncertainty. We then go one step further and study the more complex two-stage or adjustable case where firms first observe the realizations of uncertainty and can then adjust their production choices that are bounded by the originally made investments. The adjustable case is typically less conservative than the one-stage problem. 
We \rev{derive analytical results and} compare equilibrium outcomes with the corresponding robust central planner solution where the worst case sum of all profits is maximized. For our analysis we consider the case of perfectly competitive firms which act as price takers. 
We establish existence  of the resulting robust equilibrium problems and characterize them.

As we find, the market equilibrium for the perfectly competitive firms differs from the solution of the robust central planner, which is in stark contrast to classical results regarding the efficiency of market equilibria with perfectly competitive firms, compare e.g. \cite{JoskowTirole2007} or \cite{Zoettl:2010}.
For the different scenarios considered we furthermore are able to determine the resulting price of anarchy (PoA), see e.g. \cite{Koutsoupias1999} or \cite{Dubey1986} for seminal contributions.
In the case of non-adjustable robustness, for fixed demand in every time step the price of anarchy is bounded whereas it is unbounded if the buyers are modeled by elastic demand functions.
In the case of adjustable robustness, we derive an approach to compute subsidies for the firms that lead to welfare optimal equilibria in the robust market. As a direct application of our results we consider a setting where the market participants aim to optimize their respective values at risk, instead of their worst-case production costs. Here, the values at risk of the market participants determine the uncertainty set for our robust problem. We summarize our findings in \Cref{overviewtable}.

Next, we add on the brief overview of the related literature.
For the case of deterministic or risk-neutral optimization, our framework has already been well studied in the context of the so called peak--load pricing literature, compare e.g. \cite{MurphySmeers:2005}, \cite{JoskowTirole2007},  or \cite{Zoettl:2010}. 

More recently, a large strand of contributions has developed which considers stochastic equilibrium problems. \cite{EhrenmannSmeersOR:2011} for example characterize market equilibria in the case of risk averse firms in an adjustable peak-load-pricing framework. \cite{RalphSmeers:2016} include the possibility of tradable risk, which allows to reconcile risk neutral and risk averse behavior of firms. Along similar lines, results are obtained by \cite{Philpottetal:2016}. Moreover, \cite{SchiroHobbsPang:2016} provides general existence results of competitive equilibrium under risk aversion. \cite{Pangetal:2017} analyze the case of strategically behaved firms. Finally, \cite{GERARDetal:2018} show for the case of perfectly competitive firms that market equilibria for risk averse agents are not unique.

As argued above, we protect against uncertainties in a robust manner. \revv{This} work thus contributes to the recent strand of literature which seeks to explicitly introduce robust optimization in the context of equilibrium problems. Equilibrium problems where players address uncertainties from a robustness perspective play an increasingly important role in the literature. For a seminal contribution in this context consider e.g. \cite{AghassiBertsimas:2006}. Several recent articles consider different  robust equilibrium problems under incomplete information, compare e.g.  \cite{DeyandZaman:2020}, or \cite{Franzeresetal:2019} for robust bidding in auctions.
\rev{The article~\cite{MR2879331} considers a robust optimization model for a market similar to ours in continuous time, but does not include investment decisions.}

Another area of research considers robust linear complementarity problems, short LCPs, which are suited to describe strictly robust equilbrium problems. 
In this context consider
the seminal contributions by \cite{Wu_et_al:2011} and \cite{Xie_Shanbhag:2016}, and more recently \cite{KrebsSchmidt2020OMS},  or \cite{Krebsetal2021} which introduce a less conservative nonadjustable robust approach to LCPs. 
In \cite{Biefeletal:2020}, the authors consider uncertain linear complementarity under the assumption of affinely adjustable robustness and derive characterizations for conditions under which a solution exists and is unique. However, the results apply to general LCPs and do not 
answer questions on the market equilibria considered in the present work, for example the comparison of
an equilibrium to the central planner optimum.

Finally, in recent contributions, \cite{Kramer_et_al:2018} and \cite{ccelebi2021gamma} consider robust investment and production decisions of firms in the context of electricity markets. In the setup considered, uncertainty affects only those parts of the market environment which affect all firms symmetrically (i.e. in their context common market demand is subject to uncertainties). The latter allows to obtain the equivalence of competitive market equilibrium and the corresponding system optimal benchmark. Moreover, \cite{ccelebi2021gamma} focus on the single level-nonadjustable case, whereas we consider the (usually complex but typically less conservative) adjustable setting. In the context of our setup uncertainties are affecting firms in an asymmetric way (i.e. each firm's production cost). This breaks the equivalence of equilibrium and the corresponding system optimum, as we show, however. 
 
\begin{table}[h!]
  \caption{Overview of the theoretical results (PoA: Price of Anarchy, CP: Central Planner)}\label{overviewtable}
  \small{
  \begin{tabular}{ | l | l | l |}
    \hline
    &Strict robustness&Adjustable robustness\\
    \hline
    Fixed demand&PoA bounded:&Strict and adjustable CP have\\
    \Cref{sec:fixed-demand}&\Cref{Thm:fixed:main-bound}&equal objective value: \Cref{Thm:fixed:ARobustCPEquivalence}\\
    &&\\
    &&Examples showing PoA$>1$\\
    \hline
    Elastic demand&PoA unbounded:&Strict and adjustable CP have\\
    \Cref{sec:elastic-demand}&\Cref{Thm:elastic:poaunbounded}&equal objective value: \Cref{Thm:elastic:ARobustCPEquivalence}\\
    &&\\
    &&Examples showing PoA$>1$\\
    &&\\
    &&Subsidies for optimal equilibrium: \Cref{Thm:elastic:SubsidiesEq}\\
    \hline
  \end{tabular}\vspace{0.3cm}}
  \end{table}

The outline of \revv{this} work is as follows. In \Cref{sec:problem-statement}, we introduce the nominal peak load pricing model. In \Cref{sec:robust-lin-opt}, we relate the optimum robust objective function value of robust linear optimization problems with polyhedral uncertainty set to the similar problem but with box uncertainty sets. The results are necessary for determining the price of anarchy in later sections. In \Cref{sec:fixed-demand}, we explain the robust peak load pricing model and use the results from \Cref{sec:robust-lin-opt} to quantify the relation between the robust central planner solution and the robust market outcome, i.e., we quantify the price of anarchy. The results are given for fixed as well as for elastic demand, see \Cref{sec:elastic-demand}, and for the non-adjustable \rev{as well as for the} adjustable problem. In \Cref{sec:var}, we apply the results to coherent and non-coherent risk measures, particularly to the value at risk and point out relations to data driven coherent risk measures. We end with conclusions in section \ref{sec:conclusion}. 


\section{Nominal Peak Load Pricing}
\label{sec:problem-statement}

Following standard electricity market models, we consider a market taking place in $T$ time periods, where $\T=\{1,...,T\}$. We assume perfect competition, i.e. all players are price takers.

On the production side of the market, there are different producers $i\in\Nc=\{1,...,N\}$. 
In order to be able to produce the considered good (for example electricity), a producer $i$ needs to build capacity $y_i\geq 0$, where for each newly installed capacity unit, investment costs $\invc_i$ need to be paid. 
In each time period $t\in \T$, the producers decide on their production $x_{i,t}$ that may not exceed their 
respective capacities. The producers have to pay their variable costs $\varc_{i}$, 
but receive for each produced unit some exogenously given market price $\pi_t$. 
The goal of each producer is to maximize his or her profit. 
Hence, the optimization problem of producer $i$ reads
\begin{subequations}\label{eq:Producers}
\begin{align}
  \max_{x_i,y_i\geq 0}~&\sum_{t\in \T}((\pi_t-\varc_i)x_{i,t})-\invc_iy_i\\
  \st~&x_{i,t}\leq y_i,\quad\quad t\in \T.
\end{align}
\end{subequations}

We address two different cases for the demand side of the market. First, in \Cref{sec:fixed-demand}, we consider the case of fixed demand. Here, in each time period $t\in \T$ a fixed demand $d_t$ is given. 
An equilibrium $(\pi,x,y)$ is a set of prices, production choices and investments 
such that the optimality conditions of all players and the market clearing conditions
\begin{align}\label{eq:MCfix}
    &d_t=\sum_{i\in\Nc}x_{i,t}, \quad\quad t\in\T
\end{align}
are fulfilled. In the following, we denote the total production 
in time period $t$ by $\bar{x}_t$, i.e. $\bar{x}_t=\sum_{i\in \Nc}x_{i,t}$.

The problem consisting of the optimality conditions of the producer problems \eqref{eq:Producers} 
together with the market clearing condition \eqref{eq:MCfix} is equivalent to the system of necessary
and sufficient optimality conditions, in this case the \revv{Karush–Kuhn–Tucker (KKT)} conditions, of a linear program.
This problem \revv{is called} welfare optimization problem or central planner problem under fixed demand \revv{and} reads 
\begin{subequations}\label{eq:Wfix}
  \begin{align}
    \max_{x,y\geq 0}~&-\sum_{i\in \Nc}(\invc_iy_i
      +(\sum_{t\in \T}\varc_ix_{i,t}))\\
    \st~&x_{i,t}\leq y_i,\quad\quad i\in \Nc,t\in \T,\\
    &\sum_{i\in \Nc}x_{i,t}=d_t,\quad\quad t\in \T.\label{eq:WfixMarketClearing}
  \end{align}
\end{subequations}
This optimization problem minimizes the total cost in the market, i.e., it determines the system optimum that a centralized planning process would follow. 
We thus call it the central planner problem. It yields an approach to determine market clearing prizes as follows.
The optimal dual variables of \eqref{eq:WfixMarketClearing} are exactly the prizes $\pi_t$ that 
lead to a market clearing equilibrium. \rev{Further, we note that \eqref{eq:Wfix} always
has an optimal solution. Indeed, it is equivalent to a linear minimization problem that is bounded
from below by zero for which thus an optimum solution is attained.}

Second, in \Cref{sec:elastic-demand}, we consider the case of elastic demand instead of fixed demand. 
Here, an elastic demand function 
$p_t:\R\rightarrow \R$ is given in each time period. 
\rev{We assume $p_t$ to be strictly decreasing for every $t\in \T$. Note that this is one of the fundamental assumptions for most microeconomic frameworks, compare e.g. \cite{MasColell1995}.}
With this, in each time period $t\in \T$, the demand side solves the problem 
  \begin{align}\label{eq:Demandel}
      \max_{d_t\geq 0}\int_0^{d_t}p_t(s)ds - \pi_td_t
  \end{align}
The optimality conditions of the producers \eqref{eq:Producers} and the consumers \eqref{eq:Demandel}, together with the market clearing conditions \eqref{eq:MCfix} are equivalent to 
the KKT conditions of the welfare optimization problem under elastic demand 
\begin{subequations}\label{eq:Wel}
    \begin{align}
      \max_{x,y\geq 0}~&\sum_{t\in\T}\int_0^{\bar{x}_t}\rev{p_t(s)}ds
        -\sum_{i\in \Nc}(\invc_iy_i+(\sum_{t\in \T}\varc_ix_{i,t}))\\
      \st~&x_{i,t}\leq y_i,\quad\quad i\in \Nc,t\in \T.
    \end{align}
  \end{subequations}
Since the elastic demand function is strictly decreasing, the objective function in \eqref{eq:Wel} is concave. 
Thus, \rev{there exists an optimal solution and} the KKT conditions of \eqref{eq:Wel} are necessary and sufficient optimality conditions . 
This implies that optimal solution $(x^*,y^*)$ of \eqref{eq:Wel} can be extended to a market 
equilibrium by defining prices as $\pi_t=p_t(\sum_ix^*_{i,t})$ for all $t\in \T$.

In more detail, a set of prices, investments and production choices $(\pi,y,x)$ is a market equilibrium if and only if $(y,x)$ is an optimal solution of \eqref{eq:Wel}. 
The market clearing prices are then given by $\pi_t=\rev{p_t(\bar{x})}$ for all $t\in \T$. 
The vector $\bar{x}$, i.e. the total production in every time period,  is unique.
For more details, we refer to \cite{Crew1995} or Chapter 5 in \cite{schewe2019}.

\subsection*{Notation}
For the remainder of the paper, we define some notation.
We denote the unit vectors by $e_i$ and the vector of all ones by $e$ (dimension clear from context).
We use the notation $[n]=\{1,...,n\}$.
For a convex and compact set 
$\mathcal{C}\subset \R^n_{\geq 0}$ we define the parameter
\begin{subequations}\label{eq:maxtau}
  \begin{align}
    \tau(\C)~:=\quad\max_{x\in\R^n,\tau\in\R} ~&\tau\\
    \st~~~~&\tau\leq x_i,\quad\quad i\in[n],\label{eq:maxtau-c1}\\
    &x\in\C.\label{eq:maxtau-c2}
  \end{align}
\end{subequations}
\rev{Thus, $\tau(\C)$ is the largest number such that there is an $x\in \C$ with minimal entry equal to $\tau(\C)$.}


\section{Robust Linear Optimization}
\label{sec:robust-lin-opt}
Our results on the different robust market models are based on auxiliary results from robust 
linear optimization. 
To the best of our knowledge, these results are new. \revv{The results of this section are also used} to
introduce \revv{the} terminology for robust problems.

Let $A\in\R^{m\times n}$, $B\in\R^{m\times k}$, $b\in\R^m$ and let 
$\X=\{(x,y)\in\R^{n+k}_{\geq 0}\defsep Ax+By\geq b\}$. 
For cost vectors $c\in\R^n_{\geq 0}$ and 
$d\in\R^k_{\geq 0}$ we consider the linear program
\begin{align}\label{eq:NomLP}
  \min_{(x,y)\in\X} c^T x+d^T y.
\end{align}
Let us assume that the cost vector $c$ is linearly affected by uncertainty that is modeled as
$c(u)=c+\Lambda u$ for some diagonal matrix 
$\Lambda=\text{diag}(\lambda)\in\rev{\R^{n\times n}_{\geq 0}}$. 
We assume that the uncertainty \rev{parameter} $u$ lies in an arbitrary polyhedral uncertainty set of the form
\begin{align*}
  \Unc=\{u\in\R^n_{\geq 0}\defsep Pu\leq r\}\subseteq [0,1]^n,
\end{align*}
with a suitably defined matrix $P$ and \rev{vector $r$} of appropriate dimension.
Additionally, we assume that the \rev{projection of $\Unc$ onto any axis $\R e_i$ is $[0,1]$.}

In robust optimization, one seeks to find a solution that attains the minimal guaranteed objective value
among all possible realizations of the uncertainty. Such a so called robust optimum solution solves
the robust counterpart of \eqref{eq:NomLP} which reads
\begin{align}\label{eq:Rob-MMP}
  \robv:=\min_{(x,y)\in\X}~&c^T x+d^T y+\max_{u\in\Unc}(\Lambda u)^T x.
\end{align}

Instead of computing an optimal robust solution by solving \eqref{eq:Rob-MMP}, 
one could also approximate the robust counterpart by replacing the uncertainty set with the 
box $\rev{\mathcal{B}=\Lambda[0,1]^n}$. 
We obtain the problem
\begin{align}\label{eq:Box-LP}
  \boxtv:=\min_{(x,y)\in\X}~&(c+\lambda)^T x+d^T y.
\end{align}
For an optimal solution of \eqref{eq:Box-LP}, $(x^*,y^*)$, we define the objective value for the worst-case realization of the uncertainty
\begin{align*}
  \boxv:=c^T x^*+d^T y^*+\max_{u\in\Unc}(\Lambda u)^T x^*.
\end{align*}
We note that $\boxv$ depends on the choice of $x^*$. 
However, we always have $\rev{\robv\leq}~\boxv\leq \boxtv$.

We are interested in the quality of this approximation, i.e. the gap between $\boxv$ and $\robv$. 
In the following theorem we obtain a bound on this gap which will be used in the following section to compare robust market equilibria with central planner solutions.
\begin{theorem}\label{Thm:MainThm}
  The inequality 
  \begin{align*}
    \boxv\leq\frac{1}{\tauU}\robv
  \end{align*}
  holds.
\end{theorem}
\begin{proof}
  First, we note that \rev{$0< \tauU\leq 1$, where the strict inequality is due to 
  convexity of the uncertainty set $\Unc\subseteq[0,1]^n$ and the assumption that the
  projectios of $\Unc$ onto any axis $\R e_i$ is $[0,1]$.}
	Dualizing the inner maximization problem that determines the 
	worst-case uncertainty realization in \eqref{eq:Rob-MMP}, yields the linear program
	\begin{subequations}\label{eq:Rob-LP}
	  \begin{align}
	    \robv=\min_{x,y,z}~&c^T x+d^T y+r^T z\\
	    \st~&Ax+By\geq b,\\
	    &P^T r\geq \Lambda x,\\
	    &x,y,z\geq 0.
	  \end{align}
	\end{subequations}
	Its dual problem reads 
	\begin{subequations}\label{eq:Rob-DLP}
	  \begin{align}
	    \robv=\max_{\theta,u\rev{\geq 0}}~&b^T \theta\\
	    \st~&A^T \theta\leq c+\Lambda u,\\
	    &B^T \theta \leq d,\\
	    &\rev{Pu\leq r}.
	  \end{align}
	\end{subequations}
  On the other hand, the dual problem of \eqref{eq:Box-LP} is given by
  \begin{subequations}\label{eq:Box-DLP}
    \begin{align}
      \boxtv=\max_{\theta\geq 0}~&b^T \theta\\
      \st~&A^T \theta\leq c+\lambda,\\
      &B^T \theta \leq d.
    \end{align}
  \end{subequations}
  Let $\theta^*$ be an optimal solution of \eqref{eq:Box-DLP} and let 
  $(\rev{u^*},\tauU)$ be an optimal solution of \eqref{eq:maxtau} for $\mathcal{C}=\Unc$.
  Then,
  \begin{align*}
    A^T(\tauU \theta^*)
    \leq \tauU (c+\lambda)
    \leq c+\tauU\lambda\leq c+\Lambda u^*
  \end{align*}
  and
  \begin{align*}
    B^T(\tauU \theta^*)
    \leq \tauU d
    \leq d.
  \end{align*}
  Thus, $(\tauU\theta^*,u^*)$ is feasible for \eqref{eq:Rob-DLP}. Hence,
  \begin{align*}
     \robv\geq b^T (\tauU\theta^*)
     =\tauU b^T \theta^*
     =\tauU\boxtv
     \geq \tauU\boxv.
  \end{align*}
\end{proof}
\begin{remark}\label{Remark:MainThm}
Since any convex body that lies in the $[0,1]^n$ box \revv{can be approximated} with a 
polytope arbitrarily well, this inequality also holds true for closed and convex 
but non-polyhedral uncertainty sets.
In detail, let $(u^*,\tauU)$ be an optimal solution of \eqref{eq:maxtau} for $\mathcal{C}=\Unc$.
Let $\Unc^{\text{in}}=\text{conv}\{0,u^*\}\subset \Unc$ and let $\Unc^{\text{out}}\subset [0,1]^n$ 
be a polyhedral outer approximation of $\Unc$ with $u^*\in \Unc^{\text{out}}$. 
By replacing $\Unc$ with $\Unc^{\text{in}}$ in the computation of $\robv$ and with $\Unc^{\text{out}}$
in the computation of $\boxv$, one can conduct the above proof in a similar fashion.
\end{remark}

In addition to comparing market equilibria with central planner solutions, we will also compare
different robust optimization techniques applied to the central planner.
A, in general, less conservative model for robust optimization is the adjustable robustness where some of the variables are so called \textit{wait-and-see} variables that may 
be chosen after the uncertainty is revealed. 
In our case, we assume that $x$ may be chosen after the uncertainty $u$ is revealed. 
The adjustable robust counterpart of the uncertain problem then can be stated as
\begin{subequations}\label{eq:ARob-MMP}
\begin{align}
  \ARv := \min_{y\geq 0}~\max_{u\in\Unc}~\min_{x\geq 0}~&(c+\Lambda u)^T x
    +d^T y\\
  \st~&Ax+By\geq b
\end{align}
\end{subequations}
\rev{In \eqref{eq:Rob-MMP} and \eqref{eq:ARob-MMP}, only the objective functions are 
uncertain.}
Thus, Theorem 14.2.4 in \cite{Ben-Tal_et_al:2009} directly yields the following
\begin{theorem}\label{Thm:TwoStageRob}
  Let $\X$ be a \rev{convex polytope}. The equality
  \begin{align*}
    \robv=\ARv
  \end{align*}
  holds.
\end{theorem}
Therefore, in our case adjustable robustness does not reduce the price of robustness of the robust solution under the mild assumption that $\X$ is compact. 
In words, this result says that it makes no difference whether 
first one has to choose $x$ and $y$ anticipating the worst case realization of the uncertainty, or whether one 
is allowed to adapt the decisions $x$ after the uncertainty manifests itself. In particular, the second approach does not improve the robust solution values.
We slightly extend this result so that we will be able to apply it in \Cref{sec:elastic-demand}. 
To this end, we let $g:\R^n\to \R$ be a convex \rev{and deterministic} function, which we add to the objective 
function of the uncertain problem \eqref{eq:NomLP}.
The robust counterpart then reads
\begin{align*}
  \crobv:=\min_{(x,y)\geq 0}~&g(x)+c^T x+d^T y
      +\max_{u\in\Unc}(\Lambda u)^T x\\
      \st~&Ax+By\geq b
\end{align*}
and the adjustable robust counterpart is given by
\begin{align*}
  \cARv := \min_{y\geq 0}~\max_{u\in\Unc}~\min_{x\geq 0}~&g(x)
    +(c+\Lambda u)^T x+d^T y\\
  \st~&Ax+By\geq b
\end{align*}
For this setting we now obtain an analogous result to \Cref{Thm:TwoStageRob}.
\begin{theorem}\label{Thm:TwoStageRobConv}
  Let $\X$ be \rev{a convex polytope}. The equality 
  \begin{align*}
    \crobv=\cARv
  \end{align*}
  holds.
\end{theorem}
\begin{proof}
  We adapt the proof of Theorem 14.2.4 in \cite{Ben-Tal_et_al:2009} to our setting with nonlinear 
  objective function. It is clear, that $\crobv\geq \cARv$ holds. 
  
  It remains to show that $\crobv\leq \cARv$ holds as well. 
  Without loss of generality we may assume that all variables are adjustable in the adjustable 
  robust solution, i.e. \rev{the dimension $k$ of $y$ is zero}.
  With this, we have $\X=\{x\in\R^n_{\geq 0}\defsep Ax\geq b\}$ and 
  \begin{align*}
    \crobv=&\min\{t\defsep \exists x\in\X:\forall u\in\Unc: c(u)^Tx+g(x)-t\leq 0\},\\
    \cARv=&\min\{t\defsep \forall u\in \Unc:\exists x\in X:c(u)^Tx+g(x)-t\leq 0\}.
  \end{align*}
  We assume for contradiction that $\crobv>\cARv$. 
  
  There exists $\bar{t}\in \R$ such that $\crobv>\bar{t}> \cARv$. Then, for every $x\in\X$ there exists an uncertainty realization $u_x\in\Unc$ such that $c(u_x)^Tx+g(x)-\bar{t}>0$.
  Hence, for every $x\in\X$ there exists an $\varepsilon_x>0$ and a neighbourhood $U_x\subset \R^n$ such
  that 
  \begin{align}\label{convexproofeq1}
    \forall z\in U_x: c(u_x)^Tz+g(z)-\bar{t}\geq \varepsilon_x.
  \end{align}
  Since $\X$ is assumed to be compact, we know that there exist finitely many points 
  $x^1,...,x^N\in\X$ such that $\X\subseteq \bigcup_{j=1}^N U_{x^j}$.
  We set $\varepsilon=\min_{j\in[N]}\varepsilon_{x^j}$, $u_j=u_{x^j}$ and define the functions
  \begin{align*}
      f_j(x):=c(u_j)^Tx+g(x)-\bar{t}.
  \end{align*}
  \rev{Since $g$ is convex, $f_j$, $j\in [N]$, are convex as well} and from 
  \eqref{convexproofeq1} we obtain that
  \begin{align*}
    \max_{j\in[N]}f_j(x)\geq \varepsilon >0\quad\fa x\in \X.
  \end{align*}
  From Proposition 2.18 in \cite{tuy1998convex} it follows
  that there exist $\lambda_j\geq 0$, $\sum_{j\in[N]} \lambda_j =1$ such that 
  \begin{align}\label{convexproofeq2}
    f(x):=\sum_{j\in[N]} \lambda_j f_j(x)\geq \varepsilon >0 \quad\fa x\in \X.
  \end{align}
  On the other hand, let $\hat{u}:=\sum_{j\in[N]}\lambda_ju_j\in\Unc$.
  Since $\bar{t}>\cARv$, there exists an $\hat{x}\in\X$ such that 
  \begin{align*}
    c(\hat{u})^T\hat{x}+g(\hat{x})-\bar{t}\leq 0.
  \end{align*}
  However, since $c(u)$ is affine in $u$, we have 
  $c(\hat{u})^T\hat{x}+g(\hat{x})-\bar{t}=\sum_{j\in[N]}\lambda_jf_j(\bar{x})$
  and thus $f(\hat{x})=\sum_{j\in[N]}\lambda_jf_j(\hat{x})\leq 0$, contradicting \eqref{convexproofeq2}.

\end{proof}

In this section, we gathered some theoretical results on robust optimization. 
These results will be used in order to prove statements on different robust market models in the remainder 
of this paper.


\section{Robust Peak Load Pricing with Fixed Demand}\label{sec:fixed-demand}

In the peak load pricing model, different input parameters may be uncertain, for example
uncertainties in the demand. 
In the following, however, we consider the variable costs of the producers 
to be uncertain. 

As a standard assumption in robust optimization, we \rev{consider the uncertainty set $\Unc\subseteq [0,1]^{N\times T}$ to be a convex polytope}.
\rev{The uncertain variable costs are then defined by $\varc_{i,t}(u)=\varc_i+a_iu_{i,t}$ with nominal values $\varc_i$, the uncertainty parameter $u\in\Unc$, and some 
scaling factors $a_i\geq 0$ for all $i\in\Nc$. }
In order to ease the exposition, we assume that the nominal values \rev{$\varc_i$} and the worst case values \rev{$\varc_i+a_i$} are feasible realizations of the uncertainty set. 
Therefore, without loss of generality, we require $0\in\Unc$ and that the projections of $\Unc$ on \rev{any coordinate axis} is $[0,1]$, 
i.e. for every $i\in \Nc$, $t\in\T$ there exists a $u\in\Unc$ such that $u_{i,t}=1$.
Furthermore, we assume that $\Unc$ is independent of the time period such that it can be written as a \rev{Cartesian} product of the form $\Unc=\rev{\prod_{t\in\T}}\Unc'$ for 
some $\Unc'\subseteq [0,1]^{N}$.

We start by discussing the strict robustness and the extension to adjustable robustness for fixed demand in this section.


\subsection{Single Stage}

In the classical strict robust approach, all variables are here-and-now decisions 
that have to be made before the uncertainty realizes.
Hence, in the corresponding robust market problem, all producers produce according to their respective worst case variable costs. 
Thus, a robust producer solves the problem
\begin{subequations}\label{eq:RobustEq}
  \begin{align}
  \max_{x_i,y_i\geq 0}~&\sum_{t\in \T}((\pi_t-(\varc_i+a_i))x_{i,t})
    -\invc_iy_i\\
  \st~&x_{i,t}\leq y_i,\quad\quad t\in \T.
  \end{align}
\end{subequations}
Such a market with strict robust producers is equivalent to a nominal market with changed variable costs $\hat{c}^{\text{var}}_{i,t}:=\varc_{i,t}+a_i$ for all $i\in\Nc$ and $t\in\T$. 
As described in \Cref{sec:problem-statement}, we compute the equilibrium of this market 
problem by solving the equivalent optimization reformulation
\begin{subequations}\label{eq:RobustEqCP}
  \begin{align}
    \min_{x,y\geq 0}~&\sum_{i\in \Nc}(\invc_iy_i
      +(\sum_{t\in \T}(\varc_i+a_{i,t})x_{i,t}))\\
    \st~&x_{i,t}\leq y_i,\quad\quad i\in \Nc,t\in \T,\\
    &\sum_{i\in \Nc}x_{i,t}=d_t,\quad\quad t\in \T,\label{eq:RobustEqCP-MC}
  \end{align}
\end{subequations}
where the optimal dual variables of \eqref{eq:RobustEqCP-MC} are exactly the 
market clearing prices for \eqref{eq:RobustEq}. 
\rev{Analogously to the nominal case, \eqref{eq:RobustEqCP} has an optimal solution.
Let $(x^*,y^*)$ be such} an optimal solution of \eqref{eq:RobustEqCP}, and therefore the 
capacity and production choices in an equilibrium of the 
market with robust producers. We define the worst case total cost of this equilibrium by 
\begin{align*}
  E_R:=\max_{u\in\Unc} \sum_{i\in \Nc}(\invc_iy^*_i
    +(\sum_{t\in \T}(\varc_i+a_iu_{i,t})x^*_{i,t})).
\end{align*}

In the uncertain market, each player independently deals with his or her worst case scenario. 
However, it is also of interest to compare the result with that of the robust central planner problem. The solutions may differ when taking into account possible correlation between the uncertainties. The robust central planner or robust total cost minimization problem reads
\begin{subequations}\label{eq:RobustW}
  \begin{align}
    \RobWel:=\min_{x,y\geq 0}~&\max_{u\in\Unc}~\sum_{i\in \Nc}(\invc_iy_i
      +(\sum_{t\in \T}(\varc_i+a_iu_{i,t})x_{i,t}))\\
    \st~&~x_{i,t}\leq y_i,\quad\quad i\in \Nc,t\in \T,\\
    &\sum_{i\in \Nc}x_{i,t}=d_t,\quad\quad t\in \T.
  \end{align}
\end{subequations}

We now discuss the differences between the market with robust producers and the 
robust central planner problem.
In the special case \rev{that the all-ones vector $e$ is contained in $\Unc$, for example
if the uncertainty set is a box,} all prices can attain their worst case value at the same time, and the robust solution needs to protect in particular against this realization of uncertainties.
Thus, in this case \eqref{eq:RobustW} is equivalent to \eqref{eq:RobustEqCP} and we obtain the following
equality.
\begin{theorem}
  If $e\in\Unc$, then $E_R=\RobWel$.
\end{theorem}
However, the uncorrelated situation seems quite unnatural in practical situations. 
Thus, typically one has the situation in which the objective value $\RobWel$ of the
robust central planner problem \eqref{eq:RobustW} is smaller than the worst case total 
cost in the market with robust producers $E_R$. It is of interest to quantify this difference. 
It will turn out in the following that the difference only depends on the size and structure of the 
chosen uncertainty set~$\Unc$.

\begin{theorem}\label{Thm:fixed:main-bound}
  The inequalities 
  \begin{align*}
    \RobWel\leq E_R\leq \frac{1}{\tauU}\RobWel
  \end{align*}
  hold. 
  In particular, this bound is sharp in the sense that for any $\delta>0$ there exist instances such that
  \begin{align*}
    E_R \geq  \frac{1-\delta}{\tauU} \RobWel.
  \end{align*}
\end{theorem}
\begin{proof}
  Since every feasible solution of \eqref{eq:RobustEqCP} is feasible for \eqref{eq:RobustW}, it directly follows $\RobWel\leq E_R$ as the objective functions of $\RobWel$ and $E_R$ coincide for fixed $(x,y)$.
  The second inequalitiy, $E_R\leq \frac{1}{\tauU}\RobWel$, directly follows from \Cref{Thm:MainThm}.
  \\We now present an instance proving the third inequality for arbitrary $\delta\in(0,1)$.
  To this end, let $T=1$ and $d=d_1=1$. Omitting the index $t$,
  we define $\invc_i=0$ for all $i\in\Nc$, $\varc_1(u)=(1-\delta)u_1$, 
  and $\varc_i(u)=u_i$ for all $i\in\Nc\setminus\{1\}.$
  Let the polyhedral uncertainty set $\Unc$ be given by
  \begin{align*}
    \Unc=\{u\in\R^n_{\geq 0}\defsep Pu\leq r\}
  \end{align*}
  with $P\in\R^{m\times N}$ and $r\in\R^m$. Since $0\in\Unc$, w.l.o.g. we may rescale $P$ and 
  assume $r\in\{0,1\}^m$.
  In the game with robust producers, the unique equilibrium is given by $\pi=1-\delta$ and $y_1=x_1=1$. 
  Thus, the total cost of this equilibirum is $E_R=1-\delta$ in the worst case.

  In the following, we establish an upper bound for $\RobWel$:
  Inserting the parameters, we obtain
  \begin{align*}
    \RobWel=\min_{x\geq 0}~&\max_{u\in\Unc}~(1-\delta)u_1x_1+\sum_{i=2}^Nu_ix_i\\
    \st~&\sum_{i\in \Nc}x_{i}=1.
  \end{align*}
  Dualizing the inner maximization problem
  \begin{align*}
    \max_{u\in\Unc}~(1-\delta)u_1x_1+\sum_{i=2}^Nu_ix_i&\quad~
  \end{align*}
  yields
  \begin{align*}
    \quad\min_{z\geq 0}~&\sum_{i\in[m]}r_iz_i\\
    \st~&(P^T z)_1\geq (1-\delta)x_1\\
    &(P^T z)_i\geq x_i,\quad\quad i=2,...,N.
  \end{align*}
  Thus, it holds that 
  \begin{subequations}\label{WRLP}
    \begin{align}
      \RobWel=\min_{x,z\geq 0}~&\sum_{i\in[m]}r_iz_i\\
	  \st~&(P^T z)_1\geq (1-\delta)x_1\\
      &(P^T z)_i\geq x_i,\quad\quad i=2,...,N,\\
	  &\sum_{i\in \Nc}x_{i}=1.
    \end{align}
  \end{subequations}  
  We now construct a feasible solution $(\tilde{x},\tilde{z})$ for \eqref{WRLP}.
  Let $(\tu,\tauU)\in\R^{N+1}$ be an optimal basic solution of \eqref{eq:maxtau} for $\mathcal{C}=\Unc$ 
  which can be written as
  \begin{subequations}\label{eq:maxtauUnc}
    \begin{align}
      \tauU=\max_{u\geq 0, \tau}~&\tau\\
      \st~&\tau-u_i\leq 0,\quad\quad i\in\Nc\label{eq:maxtauUnc-c1}\\
      &Pu\leq r.\label{eq:maxtauUnc-c2}
    \end{align}
  \end{subequations}
  Let, w.l.o.g., $\tu_i=\tauU$ for $i\in[k]$ and $\tu_i>\tauU$
  for $i\in\Nc\setminus[k]$. Thus, the first $k$ constraints in
  \eqref{eq:maxtauUnc-c1} are tight. 
  Furthermore, as there are at least $N+1$ constraints of \eqref{eq:maxtauUnc} tight in the optimal basic 
  solution $(\tu,\tauU)$, there is a subset $J\subset [m]$ of rows of $P$ with $|J|\geq 1$ such 
  that $P_{J,\bdt}\tu=r_J$.
  \\As $(\tu,\tauU)$ is optimal, the target vector $e_{N+1}$ is a conical combination of the tight constraints in $(\tu,\tauU)$,
  i.e. 
  \begin{align*}
    e_{N+1}=\sum_{i\in[k]}\mu_i(e_{N+1}-e_i)~+~
    \begin{bmatrix}
    (P_{J,\bdt})^T\\0\cdots 0
    \end{bmatrix}\lambda
  \end{align*}
  with $\lambda\in\R^{|J|}_{\geq 0}$ and $\mu\in\R^k_{\geq 0}$.
  \\It follows that $\tilde{x}:=(P_{J,\bdt})^T \lambda\geq 0$ and especially 
  $\tilde{x}_i=0$ for all $i\in\Nc\setminus [k]$.
  Additionally, we observe $\sum_{i\in[k]}\mu_i=1$ and thus we also have $\sum_{i\in\Nc}\tilde{x}_i=1$. 
  We obtain
  \begin{align*}
    \lambda^Tr_J=\lambda^TP_{J,\bdt}\tu=\tilde{x}^T\tu=\sum_{i\in\Nc} \tilde{x}_i\tu_i=\sum_{i\in[k]} \tilde{x}_i\tu_i=\sum_{i\in[k]} \tilde{x}_i\tauU=\tauU.
  \end{align*}
  We define the vector $\tilde{z}\in\R^m_{\geq 0}$ by $\tilde{z}_i:=\lambda_i$ for $i\in J$ and 
  $\tilde{z}_i:=0$ for $i\in [m]\setminus J$.
  Since $\tilde{x}=(P_{J,\bdt})^T \lambda=P^T\tilde{z}$, $(\tilde{x},\tilde{z})$ is feasible for \eqref{WRLP}.
  \\Thus, $\RobWel\leq \sum_{i\in[m]}r_i\tilde{z}_i=\sum_{i\in J}r_i\tilde{z}_i=\lambda^Tr_J=\tauU$ and we obtain
  \begin{align*}
    \frac{E_R}{\RobWel}\geq \frac{1-\delta}{\tauU},
  \end{align*}
  concluding the proof.
\end{proof}
\begin{remark}
  In the literature, this factor between the market and the central planner is often referred to as
  \textit{price of anarchy}. Therefore, the result of the above theorem yields an explicit quantification of the price of anarchy in our robust single-stage market model\rev{, given by $\frac{1}{\tauU}.$ 
Thus,} this price can be controlled by appropriate choice of the uncertainty set against which protection is sought. 
\end{remark}
\begin{remark}
We see that the \rev{factor} between the robust market and the robust central planner depends on 
the parameter $\tauU$. We want to further discuss this parameter. 
Clearly, if $e\in\Unc$, we have $\tauU=1$.
Furthermore, from the definition of $\Unc$
as \rev{Cartesian} product of $\Unc'$ for every time step, we have $\tauU=\tau(\Unc')$. 
Also, from the assumptions on $\Unc'$ and $\Unc$ we obtain that $\tauU=\tau(\Unc')\geq \frac{1}{N}$.
Thus, an upper bound for the gap in \Cref{Thm:fixed:main-bound} is $N$.
\end{remark}

In the example showing tightness in the proof of \Cref{Thm:fixed:main-bound}, we set the nominal
values of the cost vectors to zero. 
However, in relevant applications the relative deviations are usually bounded by some factor.
We model this by introducing a parameter $\rho>0$ and assume 
$a_{i,t}\leq \rho c_i$ for all $i\in \Nc$. The previous theorem can be seen as the limiting case for $\rho \rightarrow \infty$.
\begin{theorem}\label{Thm:fixed:main-bound-restricteda}
  Let $\Unc$ be a compact and convex uncertainty set fulfilling the assumptions.
  Consider the set of restricted instances for which $a_{i,t}\leq \rho c_i$ 
  for all $i\in \Nc$ with some parameter $\rho>0$.
  Then,
  \begin{align*}
    E_R \leq \frac{1+\rho}{1+\rho\tauU} \RobWel.
  \end{align*}
  Again, for every uncertainty set $\Unc$ there are instances for which this bound is sharp.
\end{theorem}
\begin{proof}[Proof sketch]
  Due to the similarity with the proof of \Cref{Thm:fixed:main-bound}, we aim at brevity here. 
  The bound is shown by adapting the proof of \Cref{Thm:MainThm}. In addition, we use the   assumption that $\lambda_i\leq \rho c_i$ for all $i\in[n]$.
  
  For showing that the bound is sharp, we construct the following instance for arbitrary $\delta\in(0,1)$.
  Let $T=1$ (we omit the index $t$) and $d=1$. 
  We define $\invc_i=0$ for all $i\in\Nc$, $\varc_1(u)=1+\rho (1-\delta)u_1$, 
  and $\varc_i(u)=1+\rho u_i$ for all $i\in\Nc\setminus\{1\}.$
  Analogously to the proof of \Cref{Thm:fixed:main-bound}, we now can show that
  \begin{align*}
    \frac{E_R}{\RobWel}\geq \frac{1+\rho(1-\delta)}{1+\rho\tauU}
  \end{align*}
  on this instance.
\end{proof}
  
The previous theorems show that a robust hedging against uncertainties in the variable costs leads to
a gap between the optimal robust central planner and the worst case total cost in an equilibrium
with robust producers.
This upper bound on this gap depends on the one hand on the relative size of the deviations 
from the nominal costs and on the other hand on the geometry of the uncertainty set given by $\tauU$.
In  \Cref{sec:elastic-demand-strict}, where we assume elastic instead of fixed demand, 
we will see that this gap is not bounded anymore.


\subsection{Two-Stage}

In this section we analyze the more sophisticated approach of 
adjustable robustness. 
While the investment decision still has to be made before the uncertainty 
realizes, the allocation of the production is chosen after the uncertainty 
realizes. For many market situation this is a relatively natural setup, where long run capacity decisions have to be made here-and-now before the uncertainty manifests itself but short run decisions can be taken once uncertainty has unravelled to large extent. 
The corresponding adjustable robust central planner problem can be 
stated as follows. 
\begin{subequations}\label{eq:ARobustW}
  \begin{align}
    \ARobWel~:=~\min_{y\geq 0}~\max_{u\in\Unc}~\min_{x\geq 0}~
      &\sum_{i\in \Nc}(\invc_iy_i
      +(\sum_{t\in \T}(\varc_i+a_iu_{i,t})x_{i,t}))\\
    \st~&~x_{i,t}\leq y_i,\quad i\in \Nc,t\in \T,\\
    &\sum_{i\in \Nc}x_{i,t}=d_t,\quad t\in \T.
  \end{align}
\end{subequations}
As expected from \Cref{sec:robust-lin-opt}, this additional freedom of waiting with decisions 
until the realization of the uncertainty does not improve the objective value compared to the strict
robust central planner. 
\Cref{Thm:TwoStageRob} is directly applicable to these robust central planner problems and 
yields the following. 
\begin{theorem}\label{Thm:fixed:ARobustCPEquivalence}
  The optimal objective value of the robust central planner and the adjustable robust central planner
  are equal, i.e.
  \begin{align*}
    \RobWel=\ARobWel.
  \end{align*}
\end{theorem}

Next, we move to the adjustable robust market problem where the producers act as players who take robust decisions within an adjustable framework. We assume that the producers know exogenously given price 
functions $\pi_t(u)$ before they invest.
Hence, every adjustable robust player solves the problem
\begin{subequations} \label{eq:newARobustEq} 
  \begin{align}
    \max_{y_i\geq 0}~\min_{u\in\Unc}~\max_{x_i\geq 0}~
      &\sum_{t\in\T}\left((\pi_t(u)-\varc_i(u)) x_{i,t}\right)-\invc_i y_i\\
    \st~&~x_{i,t}\leq y_i, \quad t\in\T.
  \end{align}
\end{subequations}

A triple $(\pi(\cdot),y,x(\cdot))$ is an equilibrium of this adjustable robust market, 
if it satisfies the optimality conditions of all adjustable robust producers \eqref{eq:newARobustEq} 
and the market clearing condition \eqref{eq:MCfix}. 
In the nominal case, prices supporting an equilibrium are given by the dual variables of the market
clearing condition of the central planner problem. 
However, it is not obvious how to compute equilibrium supporting price functions in this adjustable
robust setting. 
The following examples show that the optimal dual variables of the market clearing conditions in the 
various robust central planner formulations do not necessarily give us market clearing price functions. 

\begin{example}\label{Ex:fixed:PricesReform}
  Consider a market with two time periods and $N=2$ producers which have costs of
  \begin{align*}
    \invc_1=\invc_2=1, ~\varc_{i,t}=1+u_{i,t}, \text{ for } i=1,2,~t=1,2.
  \end{align*}
  We assume an uncertainty set given by 
  ${\Unc=\{u\in\R^{2\times 2}_{\geq 0}\defsep u_{1,t}+u_{2,t}\leq 1,~t=1,2\}}$ and demands 
  of ${d_1=1}$ and ${d_2=2}$. 
  The optimal dual variables of the market clearing conditions \eqref{eq:RobustEqCP-MC} are 
  $\pi_1=2$ and $\pi_2=3$.
  If we use these prices in the market with adjustable robust producers \eqref{eq:newARobustEq}, 
  and force $y_1+y_2=2$ in order to fulfill market clearing in the second time period, 
  the producers produce more than the demand in the first time period if $u_{1,1}<1$ and $u_{2,1}<1$.
\end{example}

\begin{example}\label{Ex:fixed:PricesRobandARob}
  Consider a market in one time period with fixed demand $d=2$. 
  There are $N=2$ producers which have costs of
  \begin{align*}
    \invc_1=\invc_2=1, ~\varc_1=u_1,~\varc_1=u_2
  \end{align*}
  \rev{with the uncerainty set $\Unc=\{u\in\R^2_{\geq 0}\defsep u_1+u_2\leq 1\}.$ 
  In the problem of robust central planner \eqref{eq:RobustW}, 
  we can replace this uncertainty set by its vertices 
  $\{u^1=(1,0),u^2=(0,1), u^3=(0,0)\}$ and obtain the following equivalent reformulation of the robust central planner \eqref{eq:RobustW}:}
  \begin{subequations}\label{eq:RobustW-Example2}
  \begin{align}
    \RobWel=\min_{x,y,\tau\geq 0}~&\tau+y_1+y_2\\
    \st~&\tau\geq x_i, \quad\quad i=1,2,\\
    &x_i\leq y_i, \quad\quad i=1,2,\\
    &x_1+x_2=2.\label{eq:RobustW-Example2-MC}
  \end{align}
  \end{subequations}
  The optimal dual variable of \eqref{eq:RobustW-Example2-MC} is given by $\pi=\frac{3}{2}$. 
  However, for this market price, no producer will invest in the adjustable robust market, 
  as in the worst case the sum of investment and variable cost of each producer is $2$.
  Next, we consider the adjustable robust central planner. 

  For the same instance, the adjustable robust central planner \eqref{eq:ARobustW} can be equivalently
  reformulated to 
  \begin{subequations}\label{eq:ARobustW-Example22}
  \begin{align}
    \ARobWel=\min_{x,y,\tau\geq 0}~&\tau+y_1+y_2\\
    \st~&\tau\geq x_1(u^1),\\
    &\tau\geq x_2(u^2),\\
    &x_i(u^j)\leq y_i,\quad\quad i=1,2,~j=1,2,3,\\
    &x_1(u^j)+x_2(u^j)=2,\quad\quad j=1,2,3.\label{eq:ARobustW-Example22-MC}
  \end{align}
  \end{subequations}
  The optimal dual variables of \eqref{eq:ARobustW-Example22-MC} are 
  $\pi(u^1)=\pi(u^2)=0.75$ and $\pi(u^3)=0$.
  Choosing these dual variables as prices for the market with adjustable robust producers, 
  the producers would not invest in any capacity as the prices are too low to cover the investment 
  costs.
\end{example}

These examples show that the optimal dual variables of the market clearing conditions of the 
equivalent reformulation of the (strict) robust market \eqref{eq:RobustEqCP}, 
the robust central planner \eqref{eq:RobustW}, and the adjustable robust central 
planner \eqref{eq:ARobustW} do not support equilibria in the market with adjustable robust 
producers in general.
We conclude from this that we need to find a more sophisticated alternative computing 
prices in the adjustable robust market. 
We will be able to resolve this problem in \Cref{sec:elastic:twostage}
for the more interesting setting where demands are not fixed but given via an elastic demand function.


\section{Robust Peak Load Pricing with Elastic Demand}\label{sec:elastic-demand}

In this section, we discuss strict robustness as well as adjustable robustness for the peak 
load model with uncertain variable costs where demand is elastic. 
Whereas the case of fixed, unelastic demand is typically easier to be treated technically, 
as seen in the previous section, most standard market models indeed rely on the assumption of 
elastic demand which formally results from the utility maximization of customers. 
For a formal treatment and the microeconomic foundation of demand functions 
see e.g. \cite{Mascollell1995}.

We make the same assumptions on the uncertainty set $\Unc$ as in \Cref{sec:fixed-demand}.

\subsection{Single Stage}\label{sec:elastic-demand-strict}
As in the case of fixed demand, we first consider the strict robust approach. 
Again, all decisions have to be made before the realizations of the uncertainty are revealed. 
We assume, that in the market only the producers take care of the uncertainty, 
the demand side is unchanged to the nominal case.
Therefore, for exogenously given market prices $\pi_t$, each producer solves the same problem 
as in the case of fixed demand, i.e. the problem
\begin{align*}
  \max_{x_i,y_i\geq 0}~&\sum_{t\in \T}((\pi_t-(\varc_i+a_i))x_{i,t})
    -\invc_iy_i\\
  \st~&x_{i,t}\leq y_i,\quad\quad t\in \T
\end{align*}
maximizing their own profit in the worst case. 
As in the case of fixed demand, this market with strict robust producers is equivalent to a 
nominal market with modified variable costs $c_{i,t}:=\varc_{i,t}+a_i$. 
We can compute the market prices that lead to an equilibrium by solving the equivalent 
optimization reformulation
\begin{subequations}\label{eq:RobustEqCPel}   
  \begin{align}
    \max_{x,y\geq 0}~&\sum_{t\in\T}\int_0^{\bar{x}_t}\rev{p_t(s)}ds
      -\sum_{i\in \Nc}(\invc_iy_i+(\sum_{t\in \T}(\varc_i+a_{i,t})x_{i,t}))\\
    \st~&x_{i,t}\leq y_i,\quad i\in \Nc,t\in \T.
  \end{align}
\end{subequations}
For an optimal solution $(x^*,y^*)$ of \eqref{eq:RobustEqCPel}\rev{, which exists due to the same reasons as in the nominal case,} we define
\begin{align*}
  E'_R:=\min_{u\in\Unc} \sum_{t\in\T}\int_0^{\bar{x}^*_t}\rev{p_t(s)}ds
    -\sum_{i\in \Nc}(\invc_iy^*_i+(\sum_{t\in \T}(\varc_i+a_iu_{i,t})x^*_{i,t})).
\end{align*}
On the other hand, the robust central planner problem is given by
\begin{subequations}\label{eq:RobustWel}
  \begin{align}
    C'_R:=\max_{x,y\geq 0}~&\min_{u\in\Unc}~
      \sum_{t\in\T}\int_0^{\bar{x}_t}\rev{p_t(s)}ds-\sum_{i\in \Nc}(\invc_iy_i
      +(\sum_{t\in \T}(\varc_i+a_iu_{i,t})x_{i,t}))\\
    \st~&~x_{i,t}\leq y_i,\quad\quad i\in \Nc,t\in \T.
  \end{align}
\end{subequations}
It is clear, that $C'_R\geq E'_R$ holds. 
In the case of fixed demand, the welfare is just the negative total cost. We have shown that 
the ratio of $E_R$ to $\RobWel$ is bounded by some 
parameter which can be computed directly from the uncertainty set $\Unc$. 
However, in this section the welfare $W_R'$ is complemented by a \rev{nonlinear} term 
reflecting the produced amount. 
\revv{As shown next}, this leads to the unboundedness of the ratio, \rev{even in the case of affine demand function}.

\begin{theorem}\label{Thm:elastic:poaunbounded}
  For any $\eta>0$ there exist instances such that $C'_R>\eta E'_R$
\end{theorem}
\begin{proof}\label{example:PoaUnboundedEl}
  We construct an example with the desired property. Let $\Nc=2$ and $T=1$.
  Furthermore, let $\invc_1=\invc_2=0$ and $\varc_1=u_1$ and $\varc_2=u_2+\varepsilon$ 
  with $0<\varepsilon\ll 1$.
  The uncertainty set is given by $\Unc=\{u\in\R^2_{\geq 0}\defsep u_1+u_2\leq 1\}$.
  
  We consider \rev{an affine} demand function $\rev{p(s)}=\alpha-s$. In the market with robust producers, 
  only producer 1 produces as he is slightly cheaper in the worst case.
  Therefore, \eqref{eq:RobustEqCPel}  reads
  \begin{align*}
    &\max_{x\geq 0} \int_0^{x_1+x_2}(\alpha-s)ds-x_1-(1+\varepsilon)x_2\\
    &=~\max_{x_1\geq 0} \int_0^{x_1}(\alpha-s)ds-x_1
	~=~\max_{x_1\geq 0} (\alpha-1)x_1-\frac{1}{2}x_1^2.
  \end{align*}
	For $\alpha>1$, the maximum is attained at $x_1=\alpha-1$. Thus,
	\begin{align*}
	  E'_R=(\alpha-1)^2-\frac{1}{2}(\alpha-1)^2
	    =\frac{(\alpha-1)^2}{2}
	\end{align*}
	for $\alpha>1$ and $E'_R=0$ for $\alpha\leq 1$.
	However, the optimal robust welfare is computed by 
	\begin{align*}
	   C'_R=\max_{x\geq 0}\min_{u\in\Unc}\int_0^{x_1+x_2}(\alpha-s)ds-x_1u_1-x_2u_2,
	\end{align*}
	where we omit the neglectable factor $\varepsilon$. Due to the 
	\rev{symmetry} of $\Unc$, the optimal solution is attained for $x_1=x_2$.
	Hence,
	\begin{align*}
	  C'_R=\max_{d\geq 0}\int_0^{d}(\alpha-s)ds-\frac{1}{2}d
	  =\max_{d\geq 0}(\alpha-\dfrac{1}{2})d-\frac{1}{2}d^2.
	\end{align*}
	For $\alpha>0.5$, the maximum is attained at $d=\alpha-0.5$ and therefore 
	we obtain
	\begin{align*}
	  C'_R=\frac{(\alpha-\frac{1}{2})^2}{2}
	\end{align*}
	for $\alpha>0.5$ and $C'_R=0$ for $\alpha\leq 0.5$.
	Thus, for $\alpha>1$ we have
	\begin{align*}
	  \frac{C'_R}{E'_R}=\frac{(\alpha-\frac{1}{2})^2}{(\alpha-1)^2}.
	\end{align*}
	For $\alpha\in(0.5,1]$, there is no production in the market, 
	but in the optimal robust welfare solution there is still some small 
	symmetrical production. Hence, in this interval the ratio is infinity. 
	For $\alpha\leq 0.5$, there is no production in both settings.
\end{proof}


\subsection{Two-Stage}\label{sec:elastic:twostage}
Again, we discuss the more sophisticated approach of adjustable robustness.
The adjustable robust central planner under elastic demand can be 
stated analogously as in the corresponding section on fixed demand by 
\begin{subequations}\label{eq:ARobustWel}
      \begin{align}
        C'_{AR}:=\max_{y\geq 0}~\min_{u\in\Unc}~\max_{x\geq 0}~
          &\sum_{t\in\T}\int_0^{\bar{x}_t}\rev{p_t(s)}ds
          -\sum_{i\in \Nc}(\invc_iy_i
          +(\sum_{t\in \T}(\varc_i+a_iu_{i,t})x_{i,t}))\\
        \st~&~x_{i,t}\leq y_i,\quad\quad i\in \Nc,t\in \T.
    \end{align}
    \end{subequations}

As in the case of fixed demand, \Cref{Thm:TwoStageRobConv} directly yields the following statement.
\begin{theorem}\label{Thm:elastic:ARobustCPEquivalence}
  The optimal Robust Welfare and the optimal Adjustable Robust Welfare are 
  equal, i.e.
  \begin{align*}
    C'_{R}=C'_{AR}.
  \end{align*}
\end{theorem}

In the adjustable robust market, we assume, that the producers know exogenously given price 
functions $\pi_t(u)$ before they invest.
Hence, every adjustable robust player solves the problem
\begin{subequations} \label{eq:newARobustEqel} 
  \begin{align}
    \max_{y_i\geq 0}~\min_{u\in\Unc}~\max_{x_i\geq 0}~
      &\sum_{t\in\T}\left((\pi_t(u)-\varc_i(u)) x_{i,t}\right)-\invc_i y_i\\
    \st~&~x_{i,t}\leq y_i, \quad\quad t\in\T.
  \end{align}
\end{subequations}
\rev{A triple $(\pi(\cdot),y,x(\cdot))$ is an equilibrium of this adjustable robust market, 
if it satisfies the optimality conditions of all adjustable robust producers \eqref{eq:newARobustEqel} 
and of the consumer \eqref{eq:Demandel}.}
We first show that in general there is a gap between the welfare of an equilibrium and the 
optimum of the adjustable robust central planner by the following example.

\begin{example}
Consider an instance with only one timestep, i.e. $T=1$, and two producers, i.e. $N=2$, \rev{and an elastic demand function given by $p(s)=5-s$}.
We define investment costs $\invc_i=\frac{1}{5}$ and nominal variable costs 
$\varc_i=0$ for $i=1,2$, 
\rev{uncertainty set $\Unc=\text{conv}\{(0,0),(1,0),(0,1),(\frac{3}{4},\frac{3}{4})\}$ and $a_1=a_2=4$. Again, as in \Cref{Ex:fixed:PricesRobandARob}, it suffices to consider the four extreme points of the scaled polytope, namely}
$u^0=(0,0)$, $u^1=(4,0)$, $u^2=(0,4)$, and $u^3=(3,3)$.

For the solution of the adjustable robust central planner \eqref{eq:ARobustWel}, the worst case 
scenario is $u^3$ and the optimality condition reads
${5-(y^*_1+y^*_2) - 3 - \frac{1}{5}=0}$ which yields optimal capacities $y^*_1=y^*_2=0.9$.
The optimal solutions and objective values for these capacities are given by 
\begin{center}
\begin{tabular}{ r|c|c|c|c } 
 & $u^0$ & $u^1$ & $u^2$ & $u^3$ \\ \hline
 optimal objective value & $7.02$ & $3.74$ & $3.74$ & $1.62$\\ 
 optimal production $x^*$ & $(0.9,0.9)$ & $(0.1,0.9)$ & $(0.9,0.1)$ & $(0.9,0.9)$\\ 
\end{tabular}
\end{center}
Thus, the adjustable robust central planner has a worst case objective value of
$C'_{AR}=1.62$.

There are no market prices $\pi^*(\cdot)$ that would extend the central planner 
solution to an equilibrium of the market with adjustable robust producers.
However, for this simple example we can compute an equilibrium by hand.
In fact, an equilibrium is given by $(\pi',y',x')$ with
\begin{align*}
&\pi'=4.2, \\
&y'=(0.4,0.4),\\
&x'=(0.4,0.4), 
\end{align*}
which can be verified by simple calculations.
In case of scenario $u^3$, the welfare of this solution is given by 
\begin{align*}
  2 \cdot 0.4+2 \cdot 0.4-\frac{1}{2} \cdot 0.8^2-0.2 \cdot 0.8=0.92,
\end{align*}
which is smaller than $C'_{AR}=1.62$.
\end{example}

\noindent The question now is, how to compute equilibrium supporting price functions 
in general. 
In the case of fixed demand, we showed in different examples that the usual approach, 
\rev{using} the dual variables of the central planner, does not work. 
Similar examples can be constructed for this case of elastic demand by replacing the 
fixed demands in \Cref{Ex:fixed:PricesReform} and \Cref{Ex:fixed:PricesRobandARob} 
by suitable elastic demand functions. 

We postpone this discussion to future research and instead discuss an approach
to induce a welfare optimal equilibrium.
For this, we consider the idea of subsidizing the producers for their investment. 
In detail, additionally to the market prices $\pi_t(\cdot)$, 
each player $i\in\Nc$ receives a subsidy payment $\eta_i$ for building capacity. 
Therefore, each player solves the problem

\begin{subequations} \label{eq:newARobustEqSubel} 
  \begin{align}
    \max_{y_i\geq 0}~\min_{u\in\Unc}~\max_{x_i\geq 0}~
      &\sum_{t\in\T}\left((\pi_t(u)-\varc_i(u)) x_{i,t}\right)-(\invc_i-\eta_i) y_i\\
    \st~&~x_{i,t}\leq y_i, \quad\quad t\in\T.
  \end{align}
\end{subequations}

We now show how to compute market prices and subsidies supporting an equilibrium equal 
to the solution of the robust central planner problem \eqref{eq:RobustWel}.
For this, we follow an idea from \cite{ONeill:2005} for markets with integer variables.
Let $u\in\Unc$ be arbitrary but fixed and let $y^*$ be the optimal capacities in the 
strict robust central planner problem \eqref{eq:RobustWel}. 
We consider the problem
\begin{subequations}\label{eq:RobustWel-fixedy}
  \begin{align}
       \max_{x,y\geq 0}~
          &\sum_{t\in\T}\int_0^{\bar{x}_t}p_t(s)ds
          -\sum_{i\in \Nc}(\invc_iy_i
          +\sum_{t\in \T}\varc_{i,t}(u)x_{i,t})\\
        \st~&~x_{i,t}\leq y_i,\quad\quad i\in \Nc,t\in \T,\\
        &~y_i=y_i^*,\quad\quad i\in\Nc.
  \end{align}
\end{subequations}
The KKT-conditions of this problem are given by 
\begin{subequations}\label{eq:RobustWel-fixedy-KKT}
  \begin{align}
    p_t(\bar{x}_t)-\varc_{i,t}(u)-\mu_{i,t}+\phi_{i,t}&=0,\label{eq:KKTgradx}
      \quad\quad i\in\Nc,t\in\T,\\
    -\invc_i+\sum_t\mu_{i,t}+\chi_i+\eta_i&=0,\quad\quad i\in\Nc,\label{eq:KKTgrady}\\
    x_{i,t}-y_i&\leq 0,\quad\quad i\in\Nc,t\in \T,\label{eq:KKTcap}\\
    y_i-y_i^*&= 0,\quad\quad i\in\Nc,\label{eq:KKTystar}\\
    \mu_{i,t}(x_{i,t}-y_i)&=0,\quad\quad i\in\Nc,t\in\T,\label{eq:KKTcomplmu}\\
    \phi_{i,t}x_{i,t}&=0,\quad\quad i\in\Nc,t\in\T,\label{eq:KKTcomplx}\\
    \chi_iy_i&=0,\quad\quad i\in\Nc,t\in\T,\label{eq:KKTcomply}\\
    x,y,\mu,\phi,\chi&\geq 0.\label{eq:KKT0}
  \end{align}
\end{subequations}
We define the price functions by $\pi^*_t(u)=p_t(\bar{x}^*_t(u))$, where $x^*(u)$ is 
the optimal solution to \eqref{eq:RobustWel-fixedy} for the respective $u$. 
For every $i\in\Nc$ we compute the subsidy payment by 
\begin{subequations}\label{eq:Whyp2}
  \begin{align}
    \eta^*_i=\max_{x,y,\mu,\phi,\chi,\eta} ~&\eta_i\\
    \st ~&\exists u\in \Unc: (x,y,\mu,\phi,\chi,\eta)\text{ solve }
      \eqref{eq:RobustWel-fixedy-KKT}
  \end{align}
\end{subequations}
\begin{lemma}\label{lemma:eta-worstcase}
  Let $x^*(u)$ be the optimal solution of \eqref{eq:RobustWel-fixedy} for given $u\in\Unc$.
  Then, for every producer $i$ with $y^*_i>0$, 
  \begin{align}\label{eq:lemmaeta}
    \eta^*_i=\invc_i
      +\max_{u\in\Unc}\left\{\sum_{t\in\T_{i,u}}(\varc_{i,t}(u)-p_t(\bar{x}^*_t(u)))\right\}
  \end{align}
  with $\T_{i,u}(x^*)=\{t\in\T\defsep x^*_{i,t}(u)>0\}$.
\end{lemma}
\begin{proof}
  From \eqref{eq:KKTgrady} and \eqref{eq:KKTcomply} we have 
  ${\eta_i=\invc_i-\sum_t\mu_{i,t}}$. 
  Condition \eqref{eq:KKTgradx} gives us
  \begin{align*}
    \mu_{i,t}=p_t(\bar{x}_t)-\varc_{i,t}(u)+\phi_{i,t},
      \quad i\in\Nc,t\in\T.
  \end{align*}
Thus, \eqref{eq:Whyp2} is equivalent to 
  \begin{align*}
    \invc_i-\min_{u,x,y,\mu,\phi,\chi,\eta}~
      &\sum_{t\in\T}\left(p_t(\bar{x}_t)
      -\varc_{i,t}(u)+\phi_{i,t}\right)\\
    \st~&(x,y,\mu,\phi,\chi,\eta)\text{ solve }\eqref{eq:RobustWel-fixedy-KKT}\\
      &u\in\Unc.
  \end{align*}
\end{proof}
\begin{theorem}\label{Thm:elastic:SubsidiesEq}
  An equilibrium of the market with adjustable robust producers \eqref{eq:newARobustEqSubel}
  and consumer \eqref{eq:Demandel} is given by $(\pi^*(\cdot),\eta^*,y^*,x^*(\cdot))$.
\end{theorem}
\begin{proof}
We prove that $(\pi^*(\cdot),\eta^*,y^*,x^*(\cdot))$ is an equilibrium by showing that
no producer $j$ has an incentive to deviate from $y^*_j$ or $x^*_j(\cdot)$.

First, we show that given $\pi^*(\cdot),\eta^*$, and $y^*$, a producer $j$ with $y^*_j>0$ 
would not produce differently than $x^*_j$. 
Let $u$ be arbitrary but fixed and let $t\in \T$.
From \eqref{eq:KKTgradx}, \eqref{eq:KKTcomplmu}, \eqref{eq:KKTcomplx}, and the
nonnegativity of $x^*,y^*,\mu^*$, and $\phi^*$ we obtain
\begin{align}\label{eq:Subsidies-x}
    x^*_{j,t}=
      \begin{cases}
         y^*_j & \text{ if }\quad p_t(\bar{x}^*_t) - \varc_{j,t}(u) > 0, \\ 
         0 & \text{ if }\quad p_t(\bar{x}^*_t) - \varc_{j,t}(u) < 0, \\
         \text{arb.}\in [0,y^*_j] & \text{ if }\quad p_t(\bar{x}^*_t) - \varc_{j,t}(u) =0.
      \end{cases}
\end{align}
Since $\pi^*_t(u)=p_t(\bar{x}^*_t)$, this implies
\begin{align*}
  (\pi^*_t(u)-\varc_{j,t}(u))x'_{j,t}\leq (\pi^*_t(u)-\varc_{j,t}(u))x^*_{j,t}
\end{align*}
for any $0\leq x'_j\leq y^*_j$ with $x'_{j,t}(u)\neq x^*_{j,t}(u)$. 
Thus, by deviating from $x^*_j$, producer $j$'s profit cannot be increased.\\

\noindent Next, we show that the worst case profit of every producer is zero. 
Clearly, this holds true for all producers with zero capacity. 
Let $i$ be a producer with $y^*_i>0$. Their objective value in scenario $u\in \Unc$,
which we denote by $P^*_i(u)$, is given by
\begin{align*}
  P^*_i(u)=\sum_{t\in\T}\left((\pi^*_t(u)-\varc_i(u)) x^*_{i,t}\right)
    -(\invc_i-\eta^*_i) y^*_i
\end{align*}
From \eqref{eq:Subsidies-x} it follows that $x^*_{i,t}$ is either $0$ or $y^*_i$ or it
does not influence the objective value. 
Thus, with 
$\T_{i,u}=\{t\in\T\defsep p_t(\bar{x}^*_t) > \varc_{j,t}(u)\}$, we have
\begin{align*}
  P^*_i(u)=\sum_{t\in\T_{i,u}}\left((\pi^*_t(u)-\varc_i(u)) y^*_i\right)
    -(\invc_i-\eta^*_i) y^*_i
\end{align*}
Thus, for the worst case profit we have
\begin{align*}
  \min_{u\in\Unc}P^*_i(u)=
    \min_{u\in\Unc}\left\{\sum_{t\in\T_{i,u}}\left((\pi^*_t(u)-\varc_i(u)) 
      y^*_i\right)\right\}
    -(\invc_i-\eta^*_i) y^*_i.
\end{align*}
Inserting \eqref{eq:lemmaeta} from \Cref{lemma:eta-worstcase} for $\eta^*_i$ now yields 
\begin{align*}
  \min_{u\in\Unc}P^*_i(u)=0.
\end{align*}\ \\

\noindent Finally, we consider the case in which a producer $j$ deviates from $y^*_j$. 
First, let $u_j^{\min}\in\argmin_{u\in\Unc}P^*_j(u)$.
Now assume that producer $j$ invests in capacity $y'_j\neq y^*_j$.

The profit of producer $j$ in $u_j^{\min}$ is given by
\begin{align*}
  P'_j(u_j^{\min})
    &=\max_{0\leq x_j\leq y'_j}\{\sum_{t\in\T}((\pi_t(u_j^{\min})
      -\varc_j(u_j^{\min})) x_{j,t})\} 
      -(\invc_j-\eta^*_j) y'_j\\
    &=\sum_{t\in\T_{i,u}}\left((\pi^*_t(u_j^{\min})-\varc_j(u_j^{\min}))y'_j\right)
      -(\invc_j-\eta^*_j) y'_j
\end{align*}
Inserting \eqref{eq:lemmaeta} from \Cref{lemma:eta-worstcase} for $\eta^*_j$ yields
$P'_j(u_j^{\min})\leq 0$. 
Since $P'_j(u_j^{\min})$ is an upper bound for the worst case profit of producer $j$,
this means that deviating from $y^*_j$ does not increase the worst case profit of 
producer~$j$.
\end{proof}

This proof concludes the section on robust peak load pricing. We showed how to compute equilibria in a 
market with strict robust producers and \rev{proved the possible existence of a gap} to the corresponding robust central planner solutions.
In the adjustable robust setting, we showed that the adjustable robust central planner cannot improve
on the strict robust central planner. Finally, \revv{this last theorem showed} how to 
offer an incentive via subsidies such that the equilibrium in the market with adjustable robust producers is equal to the robust central planner solution.


\section{Consequences for optimizing Risk measures}
\label{sec:var}

\rev{Different approaches are known in optimization under uncertainty.} Robust optimization is well-suited in cases where one has little knowledge about the shape and structure of the uncertainties and additionally the market participants are very risk averse. 
\rev{In cases where full information about the underlying uncertainties is available, e.g. if the probability distribution $F$ of the uncertain variables $u_{i,t}\in \Unc$ are known, and if the stochastic problem is algorithmically tractably, we could replace the robust optimization approach in \eqref{eq:RobustEqCP} and \eqref{eq:RobustW} by a stochastic optimization problem. For \eqref{eq:RobustEqCP} with the expectation instead of the worst-case realization in the objective, this would read
\begin{align}\label{eq:Stochex}
	\min_{x,y\geq 0}~&\sum_{i\in \Nc}(\invc_iy_i
	+(\sum_{t\in \T}(\varc_i+a_{i,t}E(u_{i,t}))x_{i,t}))\\
	\st~&x_{i,t}\leq y_i,\quad\quad i\in \Nc,t\in \T,\notag\\
	&\sum_{i\in \Nc}x_{i,t}=d_t,\quad\quad t\in \T.\notag
\end{align}
Here, the expectation essentially plays the role of a risk measure for which protection is sought by the (not very risk-averse) market participants. Still, there is a close connection to robust optimization. Indeed, it has been shown in e.g. \cite{Bertsimas2009a}, that robust optimization can be used to analyze problems involving risk measures.  With the techniques presented below, we can also exploit} limited information on the (generally unknown) probability distribution $F: \R^N_{\geq 0}\rightarrow [0,1]$ of nonnegative random variables $u_i$ is available. The additional information on $F$ is going to inform the shape of the underlying uncertainty set $\Unc$ against which the producers and the central planner seek to protect themselves. 

To capture the risk aversity of the market participants we rely on a concept that is widely considered in the economic literature -- the \emph{value at risk} (VaR). \rev{The value at risk at the confidence level $1-\alpha$ is the $(1-\alpha)$-quantile of a loss function $U\in \R$, i.e. a random variable $U$ that measures the loss of a market participant. Suppose that this loss function is distributed by $F_U$, then the value at risk at a confidence level of $1-0.05$ is $\text{VaR}_{0.95}(U)=F_U^{-1}(0.95)$. In the following lines we omit $U$ and solely consider the (potentially unknown) distribution $F$. It is worth noting that information on the value at risk can be achieved without full knowledge on $F$. In particular, }in \cite{ElGhaoui2002a} the authors consider uncertain distributions and compute their worst-case value at risk without deriving the underlying distribution. 

\rev{Here, we suppose that the value at risk is known and }the central planner seeks to protect against this risk measure \rev{ instead of e.g. the expected value considered in \eqref{eq:Stochex}. As we deal with multivariate uncertainties we extend the above definition to a \emph{multivariate value at risk} (MVaR) defined} by Pr\'ekopa (\cite{Prekopa1990a} for discrete probability distributions, \cite{Prekopa2012a} for continuous distributions):
$$\mvar_{1-\alpha}\coloneqq \{u\in \R^N: F(u)=1-\alpha \}.$$
Since our random variables are nonnegative, a natural definition of the global uncertainty set $\Unc$ is
$$\Unc\coloneqq \{u\in \R^N_{\geq 0}: F(u)\leq 1-\alpha\}.$$
As a probability distribution, $F$ is monotonously increasing. Hence, on the one hand it is quasiconvex implying the convexity of its sublevel set $\{u\in \R^N: F(u)\leq 1-\alpha\}$ and consequently the convexity of $\Unc$. On the other hand the monotonicity restricts the $u_i$ from above and implies compactness of $\Unc$. Moreover, we can ensure that its projections on the unit vectors $e_i$ form the interval $[0,1]$ by rescaling $\Unc$ to $\diag(\var^1_{1-\alpha},\ldots , \var^N_{1-\alpha})^{-1} \Unc$ and thus have established the necessary conditions needed to directly apply the theory above for a predefined uncertainty set $\Unc=\{u\in \R^N_{\geq 0}: u\leq \mvar_{1-\alpha}\}$. 

\revv{The following paragraphs illustrate this} at the example of the robust peak load pricing with fixed demand from Section \ref{sec:fixed-demand}. Here, we let the random part of the variable costs $u$ be distributed according to the distribution $F$. Hence, the $u_i$ are distributed with respect to the marginal distributions $F_i$ and the producers aim to optimize the worst case uncertainty in $\Unc_i=\{u_i\in \R_{\geq 0}: F_i(u_i)\leq 1-\alpha\}$, which we will denote (in slight abuse of notation) as the marginal value at risk $\var_{1-\alpha}^i$. We obtain the following variant of \eqref{eq:RobustEq}:
\begin{subequations}
  \begin{align}
  \max_{x_i,y_i\geq 0}~&\sum_{t\in \T}((\pi_t-(\varc_i+\var_{1-\alpha}^i))x_{i,t})
    -\invc_iy_i\\
  \st~&x_{i,t}\leq y_i,\quad\quad t\in \T.
  \end{align}
\end{subequations}
Observe that the rescaled $\Unc$ satisfies our necessary assumptions and setting $a_i=\var^i_{1-\alpha}$ negates the rescaling effect on the definitions of the worst case total cost
\begin{align*}
  E_R=\max_{u\in\Unc} \sum_{i\in \Nc}(\invc_iy^*_i
    +(\sum_{t\in \T}(\varc_i+u_{i,t})x^*_{i,t}))
\end{align*}
of a market equilibrium $(x^*,y^*)$ in \eqref{eq:RobustEqCP} 
and the corresponding robust central planner problem
\begin{subequations}
  \begin{align}
    \RobWel=\min_{x,y\geq 0}~&\max_{u\in\Unc}~\sum_{i\in \Nc}(\invc_iy_i
      +(\sum_{t\in \T}(\varc_i+u_{i,t})x_{i,t}))\\
    \st~&~x_{i,t}\leq y_i,\quad\quad i\in \Nc,t\in \T,\\
    &\sum_{i\in \Nc}x_{i,t}=d_t,\quad\quad t\in \T.
  \end{align}
\end{subequations}
Thus, we obtain the following corollary of Theorem \ref{Thm:fixed:main-bound}:

\begin{corollary}
\rev{Let $\Unc=\{u\in \R^N_{\geq 0}: u\leq \mvar_{1-\alpha}\}$. Then} the following inequalities are satisfied
 \begin{align*}
    C_R\leq E_R \leq \frac{1}{\tau (\diag(\var^1_{1-\alpha},\ldots , \var^N_{1-\alpha})^{-1}\Unc)} \RobWel.
  \end{align*}
\end{corollary}
We note that although the scaling matrix $(\diag(\var^1_{1-\alpha},\ldots , \var^N_{1-\alpha})^{-1}$ and $\mvar_{1-\alpha}$ share the same underlying distribution it is not trivial to simplify the above result by exploiting this relationship. We further add that similar corollaries can also be obtained for the other results in Sections \ref{sec:fixed-demand} and \ref{sec:elastic-demand}.

One major criticism of the value at risk is that it is, although widely used, not a coherent risk measure, i.e. it lacks the desirable properties derived by Artzner \cite{Artzneretal:1999}. However, for nonnegative random variables also coherent risk measures can partly be included in our theory as they are related to (distributionally) robust optimization through the following statement (see \cite{Bertsimas2014a} or Theorem 3.1.1 in \cite{Brown2006a}):

A risk measure $\mu:\mathcal{X}\rightarrow \R$ is coherent if and only if there exists a family of probability measures $\mathcal{Q}$ such that 
$$\mu(X)=\sup_{q\in \mathcal{Q}}E_q(X)\text{ for every random variable }X\in \mathcal{X},$$
\rev{where $E_q$ denotes the expected value with respect to the probability distribution $q$.}
Thus, suppose the producers aim to optimize their respective risk measures $\mu_i:\R_{\geq 0}\rightarrow \R$ acting on the $i$-th component of the random vector $u\in [0,1]^N$, whereas the central planner aims to optimize a risk measure $\mu:\R_{\geq 0}^N\rightarrow \R$ acting on the whole random vector $u$. It is natural to assume that the choice of the risk measures and its underlying set $\mathcal{Q}$ is data driven from a set of past realizations of uncertainty $\mathcal{A}=\{\hat{u}^1,\ldots , \hat{u}^K\}$.
Similarly, as we have done for the value at risk, our goal for the remainder of this section is to follow the arguments in \cite{Brown2006a} to link the peak load pricing with respect to the risk measures $\mu$ to our results from Sections \ref{sec:fixed-demand} and \ref{sec:elastic-demand} by determining a convenient uncertainty set $\Unc$. 

For a fair comparison we assume that both, the central planner and the producer evaluate their risk based on the same set of probability distributions $\mathcal{Q}$ that is supported solely on $\mathcal{A}$ (see Assumption 3.2.1 in \cite{Brown2006a} or \cite{Bertsimas2014a}). In addition we assume $\hat{u}^1=0$. Hence, instead of \eqref{eq:RobustEq}, the producers aim to optimize
\begin{subequations}\label{eq:coherent_peak_load_pricing_game}
  \begin{align}
  \max_{x_i,y_i\geq 0}~&\sum_{t\in \T}((\pi_t-(\varc_i+\mu_i(u_i) ))x_{i,t})
    -\invc_iy_i\\
  \st~&x_{i,t}\leq y_i,\quad\quad t\in \T,
  \end{align}
\end{subequations}
which according to Theorem 3.2.1 in \cite{Brown2006a} is equivalent to
\begin{subequations}
  \begin{align}
  \max_{x_i,y_i\geq 0}\min_{u_i\in \Unc_i}~&\sum_{t\in \T}((\pi_t-(\varc_i+u_i))x_{i,t})
    -\invc_iy_i\\
  \st~&x_{i,t}\leq y_i,\quad\quad t\in \T,
  \end{align}
\end{subequations}
where, since $\hat{u}^1=0$, the considered $\Unc_i$ are \rev{intervals} that satisfy \begin{align*}
    \Unc_i & = \text{conv}\left(\sum_{j=1}^K q_j\hat{u}^j_i: q\in \mathcal{Q}\right)=\left[0,\max_{q\in \mathcal{Q}} \sum_{j=1}^K q_j\hat{u}^j_i\right]\\
    & \subseteq \text{conv} (\{\hat{u}_i^1,\ldots , \hat{u}_i^K\})=\left[0,\max_{j=1,\ldots, K}\hat{u}_i^j\right].
\end{align*}
Moreover, if we denote $m_i\coloneqq\max_{q\in \mathcal{Q}} \sum_{j=1}^K q_j\hat{u}^j_i$, we can apply the same rescaling trick as we did for the value at risk and obtain an uncertainty set 
$$\Unc=\diag(\frac{1}{m_1},\ldots, \frac{1}{m_N})\text{conv}\left(\sum_{j=1}^K q_j\hat{u}^j: q\in \mathcal{Q}\right)$$
that satisfies the necessary conditions at the beginning of Section \ref{sec:fixed-demand}, i.e., it is \rev{a polytope and its projection on any coordinate axis is $[0,1]$}. Thus, we obtain the following similar corollary of Theorem \ref{Thm:fixed:main-bound}:

\begin{corollary}
Given a set of probability distributions $\mathcal{Q}$ and the uncertainty set $\Unc=\text{conv}\left(\sum_{j=1}^K q_j\hat{u}^j: q\in \mathcal{Q}\right)$, then the following inequalities are satisfied
 \begin{align*}
    C_R\leq E_R \leq \frac{1}{\tau (\diag(\frac{1}{m_1},\ldots, \frac{1}{m_N})\Unc)} \RobWel.
  \end{align*}
\end{corollary}
Again, similar corollaries can also be derived for the other results in Sections \ref{sec:fixed-demand} and \ref{sec:elastic-demand} as the key to these applications is simply to construct a valid uncertainty set $\Unc$. Although a further investigation into the connections of risk measures to our uncertainty sets appears to be very promising, it exceeds the scope of the present article and is thus postponed to future research. \rev{To summarize, after having obtained structural insights for both single-stage and two-stage robust market problems under uncertain cost together with the corresponding implications for the value at risk, the subsequent section concludes this work with a brief summary of the results obtained here.}


\section{Conclusion}
\label{sec:conclusion}

In the present article, we considered equilibrium problems where firms in a market context maximize their profits in a robust way when selling their output. 
Our analysis brings the robust optimization perspective into the context of equilibrium problems. In our setup we first considered the single-stage or non-adjustable robust setting where firms can sell their output. We then went one step further and studied the more complex adjustable case where a part of the variables are wait-and-see decisions. 
We compared equilibrium outcomes with the corresponding centralized robust optimization problem where the sum of all profits is maximized. 

We established existence of the resulting robust equilibrium problems and also determined the solution of the corresponding robust central planner. 
We showed that the market equilibrium for the perfectly competitive firms differs from the solution of the centralized optimization problem. 
For the different scenarios considered we furthermore are able to determine the resulting price of anarchy.
In the case of non-adjustable robustness, for fixed demand in every time step the price of anarchy is bounded. 
It is unbounded if market demands for the different time periods are modelled to be elastic. 
As a direct application of \revv{the} results we considered settings where the market participants aim to optimize their respective risk measures, instead of their worst-case production costs. 
The risk measures of the market participants then determine the uncertainty set for our robust problem. 
We could show that all our results can be directly applied to these settings. 
In sum, our analysis thus provides first important insights on bringing robust optimization approaches in the context of market equilibrium problems.


\section*{Acknowledgments}
\label{sec:acknowledgements}
We are grateful to Veronika Grimm for stimulating discussions. 
This research has been performed as part of the Energie Campus
Nürnberg (EnCN) and is supported by funding of the Bavarian State
Government.
The authors thank the Deutsche Forschungsgemeinschaft for their
support within projects B06, B07 and B09 in the
Sonderforschungsbereich/Transregio 154 \enquote{Mathematical
  Modelling, Simulation and Optimization using the Example of Gas
  Networks}.

\section*{Declarations of interest} None

\bibliographystyle{plainnat}

\end{document}